\documentclass[a4paper, 11pt]{amsart}
\pdfoutput = 1

\usepackage{amsfonts,amsmath,amssymb,mathtools,microtype}
\usepackage{geometry}
\usepackage{amsthm}
\usepackage[hyphens]{url} 
\usepackage[linktoc = all, hidelinks, colorlinks, unicode=true]{hyperref} 
\usepackage[capitalise, compress, nameinlink, noabbrev]{cleveref} 
\hypersetup{linkcolor={blue!70!black}, citecolor={green!70!black}, urlcolor={blue!70!black}}
\usepackage[usenames, dvipsnames]{xcolor}
\usepackage{enumitem}
\usepackage{comment}
\usepackage{microtype}
\usepackage[UKenglish]{isodate}
\cleanlookdateon

\newtheorem{theorem}{Theorem}
\newtheorem{definition}[theorem]{Definition}
\newtheorem{lemma}[theorem]{Lemma}
\newtheorem{corollary}[theorem]{Corollary}
\newtheorem{question}[theorem]{Question}

\newtheorem{proposition}[theorem]{Proposition}
\newtheorem{conjecture}[theorem]{Conjecture}

\renewcommand{\ge}{\geqslant}
\renewcommand{\geq}{\geqslant}
\renewcommand{\le}{\leqslant}
\renewcommand{\leq}{\leqslant}
\renewcommand{\emptyset}{\varnothing}

\DeclareMathOperator{\KG}{\textup{\textsf{KG}}}
\DeclareMathOperator{\SG}{\textup{\textsf{SG}}}
\DeclareMathOperator{\gap}{\textup{\textsf{gap}}}
\DeclarePairedDelimiter{\size}{\lvert}{\rvert}
\DeclarePairedDelimiter{\ceil}{\lceil}{\rceil}
\DeclarePairedDelimiter{\floor}{\lfloor}{\rfloor}

\newcommand{\chif}{\chi_f}
\newcommand{\prob}{\mathbb{P}}
\newcommand{\expe}{\mathbb{E}}

\newcommand{\cO}{\mathcal{O}}

\newcommand{\defn}[1]{\textcolor{Maroon}{\emph{#1}}}

\author[Gir\~{a}o]{Ant\'{o}nio Gir\~{a}o}
\address[Gir\~{a}o, Michel, Savery]{Mathematical Institute, University of Oxford, UK}
\email{\{girao,michel,savery\}@maths.ox.ac.uk}

\author[Hendrey]{Kevin Hendrey}
\address[Hendrey]{Institute for Basic Science, Daejeon, South Korea}
\email{kevinhendrey@ibs.re.kr}

\author[Illingworth]{Freddie Illingworth}
\address[Illingworth]{Mathematical Institute, University of Oxford, UK}
\curraddr{Department of Mathematics, University College London, UK}
\email{f.illingworth@ucl.ac.uk}

\author[Lehner]{Florian Lehner}
\address[Lehner]{Department of Mathematics, University of Auckland, New Zealand}
\email{florian.lehner@auckland.ac.nz \vspace{-8pt}}

\author[Michel]{Lukas Michel}

\author[Savery]{Michael Savery}
\address[Savery]{Heilbronn Institute for Mathematical Research, Bristol, UK}

\author[Steiner]{Raphael Steiner}
\address[Steiner]{Institute of Theoretical Computer Science, ETH Z\"{u}rich, Switzerland}
\email{raphaelmario.steiner@inf.ethz.ch}

\date{\today}

\begin{document}
\title{Chromatic number is not tournament-local}
\thanks{A.G.\ and F.I.\ were supported by EPSRC grant EP/V007327/1. F.I.\ and F.L.\ were supported by MATRIX-Simons travel grants. L.M.\ and M.S.\ were supported by MATRIX travel grants. R.S.\ was supported by an ETH Z\"{u}rich Postdoctoral Fellowship. K.H.\ was supported by the Institute for Basic Science (IBS-R029-C1).}

\begin{abstract}
    Scott and Seymour conjectured the existence of a function $f \colon \mathbb{N} \to \mathbb{N}$ such that, for every graph $G$ and tournament $T$ on the same vertex set, $\chi(G)\ge f(k)$ implies that $\chi(G[N_T^+(v)])\ge k$ for some vertex $v$. In this note we disprove this conjecture even if $v$ is replaced by a vertex set of size $\cO(\log{\size{V(G)}})$. As a consequence, we answer in the negative a question of Harutyunyan, Le, Thomass\'{e}, and Wu concerning the corresponding statement where the graph $G$ is replaced by another tournament, and disprove a related conjecture of Nguyen, Scott, and Seymour. We also show that the setting where chromatic number is replaced by degeneracy exhibits a quite different behaviour.
\end{abstract}

\maketitle

\section{Introduction}

\noindent The question of what structures must appear in graphs of large chromatic number is one of the most fundamental in modern graph theory. One obvious reason for a graph to have high chromatic number is the presence of a large clique, but constructions from the 1940s and 50s of, for example, Tutte~\cite{Desc54} and Zykov~\cite{Z49} demonstrate the existence of triangle-free graphs of arbitrarily large chromatic number. In particular, there are graphs with arbitrarily large chromatic number in which every neighbourhood is independent (and hence $1$-colourable).

Berger, Choromanski, Chudnovsky, Fox, Loebl, Scott, Seymour, and Thomass\'{e}~\cite{BERGER20131} conjectured that the analogous phenomenon does not occur in tournaments. This was confirmed recently in a beautiful paper of Harutyunyan, Le, Thomass\'{e}, and Wu~\cite{tournaments} in which they showed that for every $k$ there exists an $f(k)$ such that every tournament with chromatic number\footnote{The \defn{chromatic number}, $\chi(T)$, of a tournament $T$ is the least $k$ for which there is a partition of $V(T)$ into $k$ parts each of which induces a transitive (acyclic) subtournament of $T$.} at least $f(k)$ contains a vertex $v$ such that $\chi(T[N^{+}(v)])\geq k$. 

Separately, Scott and Seymour~\cite{barbados} (see also \cite[Conj.~7]{tournaments}) conjectured a similar result for a graph and a tournament on the same vertex set. 

\begin{conjecture}[Scott and Seymour]\label{conj:SS}
    For every positive integer $k$ there exists a $\chi$ such that, for every graph $G$ with $\chi(G) \ge \chi$ and every tournament $T$ on the same vertex set, there is a vertex $v$ such that $\chi(G[N_T^+(v)])\ge k$. 
\end{conjecture}

This conjecture is supported by the observation~\cite{barbados} that the statement holds when chromatic number is replaced by fractional chromatic number (see \cref{sec:frac_chrom} for more details).
The main result of this note is a disproof of \cref{conj:SS} for $k\geq 3$. In fact, we prove something stronger: $G$ and $T$ may be chosen such that the out-neighbourhood\footnote{The \defn{out-neighbourhood}, $N^+(S)$, of a set $S$ is $\bigcup_{v \in S} N^+(v)$. This might contain vertices of $S$.} of any set of size at most $\frac{\log{\size{V(T)}}}{2\chi^2}$ is bipartite.

\begin{theorem}\label{thm:countex}
    For every positive integer $\chi$ there are arbitrarily large $N$ for which there is a graph $G$ and a tournament $T$ on the same $N$-vertex set such that $\chi(G) = \chi$ and, for every set $U$ of at most $\frac{\log{N}}{2 \chi^2}$ vertices, $\chi(G[N_T^+(U)]) \leq 2$.
\end{theorem}

We will show that $G$ can in fact be taken to be triangle-free which will be useful for our proof of \cref{cor:T1_T2}. We make two remarks concerning the optimality of \cref{thm:countex}.
\begin{itemize}[noitemsep]
    \item It is not possible to replace 2 by 1 in the bound on the chromatic number of the out-neighbourhood, even when $U$ consists of a single vertex. Indeed, suppose that $G[N_T^+(v)]$ is independent for every vertex $v$. Let $xy$ be an edge of $G$. No out-neighbourhood of a vertex of $T$ can contain both $x$ and $y$, so $\{x, y\}$ dominates $T$. But then $G$ is 3-colourable: one colour for each of $N_T^+(x)$ and $N_T^+(y)$, and a final colour for whichever of $x$ and $y$ has not been coloured.
    
    \item The bound on the size of $U$ is very close to being best possible. Let $S$ be a dominating set of $T$ of size at most $\ceil{\log_2{N}}$ (such a set can be constructed greedily). Then $N^+(S)$ contains all vertices of $G$ except perhaps one and so, for any $0 \le \ell \le \chi - 2$, there is some $U \subseteq S$ of size at most $\ceil{\log_2(N)/\floor{\frac{\chi - 2}{\ell}}}$ with $\chi(G[N_T^+(U)]) > \ell$.
\end{itemize}

\Cref{thm:countex} has the following corollary, which resolves in a strong sense a question of Harutyunyan, Le, Thomass\'{e}, and Wu~\cite{tournaments} concerning the analogous problem for two tournaments on the same vertex set.

\begin{corollary}\label{cor:T1_T2}
    For every positive integer $\chi$ there are arbitrarily large $N$ for which there are tournaments $T_1$ and $T_2$ on the same $N$-vertex set such that $\chi(T_1) = \chi$ and, for every set $U$ of at most $\frac{\log{N}}{8 \chi^2}$ vertices, $\chi(T_1[N_{T_2}^+(U)])\le 2$. 
\end{corollary}

In turn, \cref{cor:T1_T2} has the following immediate consequence which disproves a conjecture of Nguyen, Scott, and Seymour~\cite[3.4]{NSS23}.

\begin{corollary}\label{cor:bipartite}
    For every positive integer $\chi$ there are arbitrarily large $N$ for which there is an $N$-vertex tournament $T$ and disjoint subsets $A, B \subseteq V(T)$ such that $\chi(T[A]), \chi(T[B]) \geq \chi$ and the following holds. For all $A' \subseteq A$ and $B' \subseteq B$ of size at most $\frac{\log{N}}{32 \chi^2}$, both $\chi(A \cap N^{+}(B'))$ and $\chi(B\cap N^{+}(A'))$ are at most $2$.
\end{corollary}

Finally, we include two results for the setting where chromatic number is replaced by degeneracy (or equivalently maximum average degree). Since every graph of high chromatic number has high degeneracy, \cref{thm:countex} shows that for every positive integer $d$ there is a graph $G$ and a tournament $T$ on the same vertex set such that the degeneracy of $G$ is at least $d$, but the subgraph of $G$ induced on each out-neighbourhood of $T$ is bipartite. Our next result strengthens this statement by ensuring that the graph induced on the out-neighbourhood is 1-degenerate.

\begin{proposition}\label{Prop:1}
    For every positive integer $k$, there is a $k$-regular graph $G$ and a tournament $T$ on the same vertex set such that $G[N_T^{+}(v)]$ is a forest for every vertex $v$.
\end{proposition}

Despite this result, and in contrast to \cref{thm:countex}, if $G$ has high degeneracy and $T$ is a tournament on the same vertex set, then there is a two-vertex set whose out-neighbourhood has high degeneracy.

\begin{theorem}\label{thm:avg_deg_2vtxs}
    For every positive integer $k$, every graph $G$ with degeneracy at least $12k$, and every tournament $T$ on the same vertex set, there exist vertices $x,y$ such that $G[N^{+}(\{x,y\})]$ has degeneracy at least $k - 1$. 
\end{theorem}

\section{Proofs of the main theorems}

\noindent In this section we present the proof of \cref{thm:countex}. Our construction is based on the classical \emph{Schrijver graphs}~\cite{schrijver}.

\begin{definition}
    Let $k \ge 1$ and $n \ge 2k$ be integers. The \defn{Kneser graph} $\KG(n,k)$ is the graph whose vertex set is $\binom{[n]}{k}$ and in which two distinct sets $S_1, S_2 \in \binom{[n]}{k}$ are adjacent if and only if $S_1 \cap S_2 = \emptyset$. The \defn{Schrijver graph} $\SG(n,k)$ is the induced subgraph of $\KG(n,k)$ whose vertex set consists of all stable sets in $\binom{[n]}{k}$. Here, a set $S \in \binom{[n]}{k}$ is called \defn{stable} if it does not include two cyclically consecutive\footnote{By this we mean a pair $i, i + 1$ where $1 \le i < n$ or the pair $n, 1$.} elements of $[n]$.
\end{definition}

Kneser~\cite{kneser} conjectured that the chromatic number of $\KG(n, k)$ is $n - 2k + 2$. This conjecture remained open for two decades and was first proved by Lov\'{a}sz~\cite{lovasz} using homotopy theory (see also B\'{a}r\'{a}ny~\cite{barany} and Greene~\cite{greene} for very short proofs). Shortly afterwards, Schrijver~\cite{schrijver} introduced the graphs $\SG(n,k)$ and proved that $\SG(n, k)$ is vertex-critical with chromatic number $\chi(\SG(n, k)) = \chi(\KG(n, k)) = n - 2k + 2$.

To prove \cref{thm:countex}, we will show that for every integer $\chi \ge 3$ and every sufficiently large integer $k$ there exists a tournament $T$ on the same vertex set as $\SG(2k+\chi-2,k)$ such that for every $U \subseteq V(T)$ which is sufficiently small, the out-neighbourhood of $U$ in $T$ induces a bipartite subgraph of $\SG(2k+\chi-2,k)$. As $\chi(\SG(2k+\chi-2,k))=\chi$, this will prove \cref{thm:countex}. 

In constructing our tournament, we rely on the following combinatorial statement which follows directly from the existence of tournaments with high domination number. 

\begin{lemma}\label{lemma:function}
    For every positive integer $t$ there is some $n_0$ such that for all integers $n \geq n_0$ there exists a function $f \colon \binom{[n]}{t} \to 2^{[n]}$ with the following two properties:
    \begin{itemize}[noitemsep]
        \item for every $A, B \in \binom{[n]}{t}$, at least one of $A \cap f(B)$ and $B \cap f(A)$ is empty, and
        \item for every collection $(A_i)_{i \in I}$ of at most $\frac{\log{n}}{2t}$ sets from $\binom{[n]}{t}$,
        \begin{equation*}
            \bigcap_{i \in I} f(A_i) \neq \emptyset.
        \end{equation*}
    \end{itemize}
\end{lemma}

\begin{proof}
    By a classical result of Erd\H{o}s~\cite{erdos} (see~\cite{graham} for an explicit construction), for every sufficiently large $n$ there is an $n$-vertex tournament in which every set of at most $\log(n)/2$ vertices is dominated by a vertex outside the set. Let $n$ be large enough that this result holds and that $\log(n)/2 \geq t$, and let $T$ be the corresponding tournament. Identify $V(T)$ with $[n]$ and, for $A \in \binom{[n]}{t}$, define $f(A)$ as
    \begin{equation*}
        f(A) \coloneqq \{v \in [n]\setminus A \colon v \text{ dominates } A\}.
    \end{equation*}
    We claim $f$ satisfies the two properties of the lemma statement. Firstly, let $A, B \in \binom{[n]}{t}$ and suppose for a contradiction that $A \cap f(B)$ and $B \cap f(A)$ are both non-empty. Then there is some $a \in A\setminus B$ that dominates $B$ and some $b \in B\setminus A$ that dominates $A$. This implies that $a$ and $b$ are distinct, and the edge between them is oriented in both directions, which is a contradiction. Next, let $(A_i)_{i \in I}$ be a collection of at most $\frac{\log{n}}{2t}$ sets from $\binom{[n]}{t}$. Let $A = \bigcup_{i \in I} A_i$ which is a set of size at most $\log(n)/2$. By the definition of $T$ some vertex $x\not\in A$ dominates $A$, but then $x\in \bigcap_{i \in I} f(A_i)$, as required.
\end{proof}

Before giving the proof of \cref{thm:countex}, let us fix the following notation: for a set $S \in \binom{[n]}{k}$, we denote by $\gap(S)$ the set of ``left-elements'' of cyclically consecutive pairs of $[n]$ that are disjoint from $S$. Concretely, $r \in \gap(S)$ if and only if $\{r, r + 1\} \cap S = \emptyset$, where addition is to be understood modulo $n$ (that is, $n + 1$ is identified with $1$). Pause to note that every stable set $S\subseteq [n]$ of size $k$ (that is, every vertex of the Schrijver graph $\SG(n,k)$) satisfies $\size{\gap(S)} = n - 2k$. Every $S \in \binom{[n]}{k}$ can be recovered from $\gap(S)$ and so $\size{V(\SG(n, k))} \leq \binom{n}{n - 2k}$. 

\begin{proof}[Proof of \cref{thm:countex}]
    The result is trivial for $\chi\leq 2$, so let $\chi \geq 3$ be an integer, $t \coloneqq \chi - 2$, and $n_0$ be as given by \cref{lemma:function}. Pick some positive integer $k > t$ such that $2k + t \ge n_0$, set $n \coloneqq 2k + t$, and set $G \coloneqq \SG(n, k)$. Note that $G$ is triangle-free, has chromatic number $\chi$ and, for any $S \in V(\SG(n, k))$, $\gap(S) \in \binom{[n]}{t}$. Hence, $N \coloneqq \size{V(\SG(n, k))} \leq \binom{n}{t} \leq n^t$.
    
    Let $f \colon \binom{[n]}{t} \to 2^{[n]}$ be the function from \cref{lemma:function}. Define a directed graph $D$ on the same vertex set as $G$ that has a directed edge from a vertex $S_1$ to a vertex $S_2$ if and only if $f(\gap(S_1)) \cap \gap(S_2) = \emptyset$. 
    Note, by the first property of $f$ guaranteed by \cref{lemma:function}, that any two distinct vertices of $D$ are connected by an arc in at least one of the two possible directions. Hence, there exists a spanning subdigraph $T$ of $D$ which is a tournament. 
    
    Let $U$ be any set of at most $\frac{\log{N}}{2 \chi^2} \leq \frac{\log{N}}{2 t^2} \leq \frac{\log{n}}{2t}$ vertices. To finish the proof we will show that the out-neighbourhood $N_D^+(U)$ induces a bipartite subgraph of $G$ (and hence the same is true for the out-neighbourhood $N_T^+(U) \subseteq N_D^+(U)$ in $T$). Write $U = \{S_1, \dotsc, S_{\size{U}}\}$. By the second property of $f$ guaranteed by \cref{lemma:function}, there is some $r \in [n]$ common to all the $f(\gap(S_i))$. By the definition of $D$, any $S \in N_D^+(U)$ satisfies $r \notin \gap(S)$ and so $S \cap \{r, r + 1\} \neq \emptyset$. Colouring all the vertices in the out-neighbourhood that include the element $r$ with one colour and all the remaining vertices (which necessarily contain $r + 1$) with another colour provides a proper $2$-colouring of $G[N_D^+(S)]$. This concludes the proof of the theorem.
\end{proof}

We can convert the graph $G$ from \cref{thm:countex} to a tournament: pick any linear order on the vertices of $G$ and construct a tournament $T_1$ whose back-edge graph is $G$. We will show that $\chi(G)$ and $\chi(T_1)$ are closely related, and thus prove \cref{cor:T1_T2}.

\begin{proof}[Proof of \cref{cor:T1_T2}]
    Let $K \coloneqq 2 \chi$ and $n$ be sufficiently large. By \cref{thm:countex} there is a triangle-free graph $G$ with chromatic number $K$ and a tournament $T$ on the same $N$-vertex set such that, for every set $U$ of at most $\frac{\log{N}}{8\chi^2}$ vertices, $\chi(G[N_T^+(U)])\le 2$. Let $(V(G), \prec)$ be a linear order and define a tournament $T_1$ with vertex set $V(G)$ as follows: there is an arc from vertex $u$ to vertex $v$ in $T_1$ if either $v \prec u$ and $uv \in E(G)$ or $u \prec v$ and $uv \notin E(G)$. We further set $T_2 \coloneqq T$ and claim that the pair $(T_1, T_2)$ of tournaments satisfies the statement of the corollary.

    Let $W \subseteq V(G)$ be any set of vertices where $T_1[W]$ is transitive. Note that if $v_1 v_2 v_3$ is a path in $G$ (so $v_1 v_3 \notin E(G)$ by triangle-freeness) and $v_1 \prec v_2 \prec v_3$, then $v_1 v_2 v_3$ is a cyclic triangle in $T_1$ and so $v_1$, $v_2$, $v_3$ are not all in $W$. In particular, the partition $W = W_1 \cup W_2$ where
    \begin{align*}
        W_1 & \coloneqq \{w \in W \colon \text{there is $w' \in W$ such that $w' \prec w$ and $w'w \in E(G)$}\}, \\
        W_2 & \coloneqq \{w \in W \colon \text{there is no $w' \in W$ such that $w' \prec w$ and $w'w \in E(G)$}\},
    \end{align*}
    gives a proper 2-colouring of the vertices of $G[W]$. Since this holds for any $W$ where $T_1[W]$ is transitive, we have $\chi(T_1) \geq \chi(G)/2 = \chi$.
    
    To finish the proof, consider any set $U$ of at most $\frac{\log{N}}{8 \chi^2}=\frac{\log N}{2K^2}$ vertices. Note that $G[N_T^+(U)] = G[N_{T_2}^+(U)]$ is bipartite. Let $I_1$, $I_2$ be two disjoint independent sets in $G$ such that $I_1 \cup I_2 = N_{T_2}^+(U)$. Now consider any two vertices $u, v \in I_j$ for some $j \in \{1,2\}$ and note that since $uv \notin E(G)$, there is an arc from $u$ to $v$ in $T_1$ if and only if $u \prec v$. Hence $T_1[I_1]$ and $T_1[I_2]$ are transitive tournaments and so $\chi(T_1[N_{T_2}^+(U)]) \le 2$. 
\end{proof}

To prove \cref{cor:bipartite}, we can now take the two tournaments $T_1$ and $T_2$ from \cref{cor:T1_T2} and combine them appropriately: we simply orient the edges within $A$ and $B$ according to $T_1$, and the edges between $A$ and $B$ according to $T_2$.

\begin{proof}[Proof of \cref{cor:bipartite}]
    Let $\chi$ be a positive integer. By \cref{cor:T1_T2}, for arbitrarily large $N$ there exist tournaments $T_1$ and $T_2$ on the same $N$-vertex set $V$ with $\chi(T_1) = 2\chi$ and $\chi(T_1[N^{+}_{T_2}(U)]) \leq 2$ for every $U \subseteq V$ of size at most $\frac{\log{N}}{32 \chi^2}$. Partition $V$ into sets $A$ and $B$ such that $\chi(T_1[A]),\chi(T_1[B])\geq \chi$, then construct a new tournament $T$ on $V$ by orienting the edge between $u,v\in V$ to agree with $T_1$ if $u,v\in A$ or $u,v\in B$, and orienting it to agree with $T_2$ otherwise. It is not difficult to see that $T$ satisfies the conditions of the corollary.
\end{proof}

\section{Degeneracy}\label{sec:degen}

\noindent In this section we consider the setting in which degeneracy replaces chromatic number. We first show that there is a tournament on the vertex set of the $k$-dimensional hypercube such that each out-neighbourhood induces a forest in the hypercube, proving \cref{Prop:1}. Therefore, having high degeneracy does not imply that some out-neighbourhood has high degeneracy.

\begin{proof}[Proof of \cref{Prop:1}]
    For each $k$, let $G_k$ be the hypercube on $2^{k}$ vertices. We will actually prove something stronger than \cref{Prop:1}, namely that the \emph{closed} in- and out-neighbourhoods\footnote{The \defn{closed in-neighbourhood} of a vertex $v$ in tournament $T$ is $N_T^-[v] \coloneqq \{v\} \cup N_T^-(v)$. The closed out-neighbourhood is defined analogously.} $G_k[N_T^{-}[v]]$ and $G_k[N_T^{+}[v]]$ are both forests for every vertex $v\in V(G_k)$. We proceed by induction on $k$. For $k=1$ the result is immediate, so given $k\geq 1$ let $T_k$ be a tournament on $V(G_k)$ with the desired property. We will view $G_{k+1}$ as the union of two copies of $G_k$, say $G_k^{1}$ and $G_{k}^{2}$, connected via the matching consisting of all edges of the form $x^{1}x^{2}$, where $x^{1}\in V(G_k^{1})$ and $x^2\in V(G_{k}^{2})$ denote the copies of a vertex $x \in V(G_k)$. For each $S\subseteq V(G_k)$, we will write $S^{(1)}$ and $S^{(2)}$ for the corresponding sets of vertices in $G_k^{1}$ and $G_{k}^{2}$ respectively.
    
    Now define a tournament $T_{k+1}$ on vertex set $V(G_{k+1})$ as follows. First orient the edges within each of $V(G_k^1)$ and $V(G_k^2)$ according to $T_k$, in the canonical way. Then for each $x\in V(G_k)$, orient every edge between $x^1$ and ${N_{T_k}^{-}[x]}^{(2)}$ away from $x^1$ and every edge between $x^1$ and ${N_{T_k}^{+}(x)}^{(2)}$ towards $x^1$. This completes the construction of $T_{k+1}$. Observe that for each $x\in V(G_k)$, the edges between $x^2$ and ${N_{T_k}^{-}(x)}^{(2)}$ are oriented away from $x^2$ and the edges between $x^2$ and ${N_{T_k}^{+}[x]}^{(2)}$ are oriented towards $x^2$.
    
    Let $x\in V(G_k)$ and note that $N_{T_{k+1}}^+[x^1]={N_{T_k}^+[x]}^{(1)}\cup {N_{T_k}^{-}[x]}^{(2)}$. By the induction hypothesis, $N_{T_k}^+[x]$ and $N_{T_k}^{-}[x]$ both induce forests in $G_{k}$, so ${N_{T_k}^+[x]}^{(1)}$ and ${N_{T_k}^{-}[x]}^{(2)}$ do the same in $G_{k+1}$. Since there is exactly one edge in $G_{k+1}$ between these two sets, namely $x^1x^2$, the graph $G_{k+1}[N_{T_{k+1}}^+[x^1]]$ is acyclic. Analogous arguments show that $G_{k+1}[N_{T_{k+1}}^{-}[x^1]]$, $G_{k+1}[N_{T_{k+1}}^{+}[x^2]]$, and $G_{k+1}[N_{T_{k+1}}^{-}[x^2]]$ are all acyclic too. Since every vertex of $G_{k+1}$ is of the form $x^1$ or $x^2$ for some $x\in V(G_k)$, this completes the proof.
\end{proof}

However, we will now show that, unlike with chromatic number, having high degeneracy implies that there are two vertices $x$ and $y$ such that the out-neighbourhood of $\{x, y\}$ has high degeneracy.

\begin{proof}[Proof of \cref{thm:avg_deg_2vtxs}]
     Let $H$ be a bipartite subgraph of $G$ with $\delta(H) \geq 6 k$ and let $A \cup B$ be a bipartition of $H$ with $\size{A} \geq \size{B}$. Define $T_1 = T[A]$ and $T_2 = T[B]$. 
     Pick $x \in A$ satisfying $\size{N_{T_1}^{+}[x]} \geq \size{A}/2$ and define $A' = N_{T_1}^{+}[x]$. 
     Now let $H_1 = H[A', B]$. It can be shown using linear programming duality that every tournament has a probability distribution on its vertex set which assigns weight at least $1/2$ to every closed in-neighbourhood (see \cite[Sec.~1.2]{fisherryan}).
     Let $w$ be such a probability distribution for $T_2$. 
     Take a random vertex $y \in B$ according to $w$ and note that $\prob(u \in N^+_{T_2}[y]) \geq 1/2$ for every $u\in B$. Let $H_2 = H_1[A', N_{T_2}^{+}[y]]$ so that for every $e \in E(H_1)$, $\prob[e \in E(H_2)] \geq 1/2$.
     We have $\expe[e(H_2)] \geq e(H_1)/2 \geq 3k \size{A'} \geq k(\size{A'} + \size{B})$, from which it follows, since $\size{N_{T_2}^{+}[y]} \leq \size{B}$, that there exists $y\in B$ such that $e(H_2) \geq k \size{V(H_2)}$. Removing $x$ and $y$ from $H_2$, we obtain a subgraph $G'$ of $G[N^+_T(\{x, y\})]$ with $e(G') \geq (k - 2) \size{V(G')}$. Thus $G'$, and therefore also $G[N^+(\{x, y\})]$, has degeneracy greater than $k - 2$.
\end{proof}

\section{Fractional chromatic number}\label{sec:frac_chrom}

\noindent We remind the reader that a graph $G$ has fractional chromatic number $\chif(G)\leq r$ if and only if there is a probability distribution on the independent sets of $G$ such that the random independent set $\boldsymbol{I}$ obtained and every vertex $v$ satisfy $\prob(v \in \boldsymbol{I}) \geq 1/r$. In this section we demonstrate that the modified version of \cref{conj:SS} in which chromatic number is replaced by fractional chromatic number is true, as observed by Scott and Seymour~\cite{barbados} without proof.

\begin{theorem}\label{thm:frac}
    For $c\geq 1$, let $G$ be a graph and $T$ be a tournament on the same vertex set such that $\chif(G[N_T^{+}(v)]) \leq c$ for every vertex $v$. Then $\chif(G) \leq 2(c + 1)$.
\end{theorem}

\begin{proof}
    Let $w$ be a probability distribution on the vertex set of $T$ that assigns weight at least $1/2$ to every closed in-neighbourhood. For each vertex $v$, since $\chif(G[N_T^+(v)]) \leq c$, there is a random independent set $\boldsymbol{I}_v$ of $G[N_T^+(v)]$ such that $\prob(u \in \boldsymbol{I}_v) \geq 1/c$ for every $u \in N_T^+(v)$.

    We sample a random independent set $\boldsymbol{I}$ of $G$ as follows. First pick a vertex $\boldsymbol{v}$ according to $w$. Then with probability $1/(c + 1)$ take $\boldsymbol{I} = \{\boldsymbol{v}\}$ and with probability $c/(c + 1)$ take $\boldsymbol{I} = \boldsymbol{I}_{\boldsymbol{v}}$. Note that, for any vertex $u$, if $\boldsymbol{v} \in N^-[u]$, then $u \in \boldsymbol{I}$ with probability at least $1/(c + 1)$. Hence, by the defining property of $w$, $\prob(u \in \boldsymbol{I}) \geq 1/(2c + 2)$ and so $\chif(G) \leq 2(c + 1)$.
\end{proof}

\section{Closing remarks}

\noindent We have been unable to determine whether high chromatic number forces an out-neighbourhood with high degeneracy, and we would be interested to know if this is the case.

\begin{question}
    Does there exist, for each integer $d$, an integer $\chi$ such that for every graph $G$ with $\chi(G)\geq \chi$ and every tournament $T$ on the same vertex set, there is a vertex $v$ for which $G[N^{+}_T(v)]$ has degeneracy at least $d$?
\end{question}

We do, however, suspect that this is true for $d = 2$, that is, it should be possible to force some out-neighbourhood to contain a cycle.

\begin{conjecture}
    For every graph $G$ with sufficiently large chromatic number, and every tournament $T$ on the same vertex set, there exists a vertex $v$ such that $G[N^{+}_T(v)]$ contains a cycle.
\end{conjecture}

We have shown that for certain very structured tournaments $T$ there are graphs on the same vertex set with large chromatic number, in which every out-neighbourhood of $T$ induces a bipartite subgraph. We conjecture that (with high probability) we cannot replace $T$ with a random tournament.

\begin{conjecture}
    For every positive integer $k$, there exists a $\chi$ such that if $T$ is the uniformly random tournament on vertex set $[N]$, then with high probability \textup{(}as $N\to\infty$\textup{)}, for every graph $G$ on $[N]$ with $\chi(G)\geq \chi$ there is a vertex $v\in[N]$ for which $G[N^+_T(v)]\geq k$.
\end{conjecture}

Finally, as remarked after the statement of \cref{thm:countex}, if $\chi(G) \geq \chi$, then there is a collection of at most $\ceil{\log_2(N)/\floor{\chi/2 - 1}}$ out-neighbourhoods whose union induces a subgraph of chromatic number at least 3. It would be interesting to know if $o(\log(N)/\chi)$ (as $\chi\to \infty$) out-neighbourhoods suffice here. In particular, we conjecture the following.

\begin{conjecture}
    There exists $f(N)$ satisfying $f(N)=o(\log N)$ such that for every $N$-vertex graph $G$ with $\chi(G)\geq f(N)$, and every tournament $T$ on the same vertex set, there is a vertex $v$ for which $\chi(G[N^+_T(v)])\geq 3$.
\end{conjecture}

\subsection*{Acknowledgements} We would like to thank Sang-il Oum, Alex Scott, David Wood, and Liana Yepremyan for organising the April 2023 \href{https://www.matrix-inst.org.au/events/structural-graph-theory-downunder-iii/}{MATRIX-IBS Structural Graph Theory Downunder III} workshop where we began this work. Our thanks to Paul Seymour for helpful comments on the paper.

{
\fontsize{10pt}{11pt}
\selectfont

\setlength{\parskip}{2pt plus 0.3ex minus 0.3ex}
	
\bibliographystyle{bib.bst}
\bibliography{refs.bib}

\newcommand{\etalchar}[1]{$^{#1}$}
\begin{thebibliography}{HLTW19}
\providecommand{\url}[1]{\texttt{#1}}
\providecommand{\urlprefix}{\textsc{url:} }
\expandafter\ifx\csname urlstyle\endcsname\relax
  \providecommand{\doi}[1]{doi:\discretionary{}{}{}#1}\else
  \providecommand{\doi}{doi:\discretionary{}{}{}\begingroup
  \urlstyle{rm}\Url}\fi

\bibitem[B{\'a}r78]{barany}
\textsc{Imre B{\'a}r{\'a}ny} (Nov. 1978).
\newblock \href{https://doi.org/10.1016/0097-3165(78)90023-7}{A short proof of
  {K}neser's conjecture}.
\newblock \emph{Journal of Combinatorial Theory, Series A} \textbf{25}(3),
  325--326.

\bibitem[BCC{\etalchar{+}}13]{BERGER20131}
\textsc{Eli Berger}, \textsc{Krzysztof Choromanski}, \textsc{Maria Chudnovsky},
  \textsc{Jacob Fox}, \textsc{Martin Loebl}, \textsc{Alex Scott}, \textsc{Paul
  Seymour}, and \textsc{St\'{e}phan Thomass\'{e}} (Jan. 2013).
\newblock \href{https://doi.org/10.1016/j.jctb.2012.08.003}{Tournaments and
  colouring}.
\newblock \emph{Journal of Combinatorial Theory, Series B} \textbf{103}(1),
  1--20.

\bibitem[Des54]{Desc54}
\textsc{Blanche Descartes} (May 1954).
\newblock $k$-chromatic graphs without triangles.
\newblock \emph{American Mathematical Monthly} \textbf{61}(5), 352--353.

\bibitem[Erd63]{erdos}
\textsc{Paul Erd\H{o}s} (Oct. 1963).
\newblock \href{https://doi.org/10.2307/3613396}{On a problem in graph theory}.
\newblock \emph{The Mathematical Gazette} \textbf{47}, 220--223.

\bibitem[FR95]{fisherryan}
\textsc{David~C. Fisher} and \textsc{Jennifer Ryan} (Mar. 1995).
\newblock \href{https://doi.org/10.1002/jgt.3190190208}{Tournament games and
  positive tournaments}.
\newblock \emph{Journal of Graph Theory} \textbf{19}(2), 217--236.

\bibitem[Gre02]{greene}
\textsc{Joshua~E. Greene} (Dec. 2002).
\newblock \href{https://doi.org/10.2307/3072460}{A new short proof of
  {K}neser's conjecture}.
\newblock \emph{The American Mathematical Monthly} \textbf{109}(10), 918--920.

\bibitem[GS71]{graham}
\textsc{Ronald~L. Graham} and \textsc{Joel~H. Spencer} (Jan. 1971).
\newblock \href{https://doi.org/10.4153/CMB-1971-007-1}{A constructive solution
  to a tournament problem}.
\newblock \emph{Canadian Mathematical Bulletin} \textbf{14}, 45--48.

\bibitem[HLTW19]{tournaments}
\textsc{Ararat Harutyunyan}, \textsc{Tien-Nam Le}, \textsc{St\'{e}phan
  Thomass\'{e}}, and \textsc{Hehui Wu} (Sep. 2019).
\newblock \href{https://doi.org/10.1016/j.jctb.2019.01.005}{Coloring
  tournaments: from local to global}.
\newblock \emph{Journal of Combinatorial Theory, Series B} \textbf{138},
  166--171.

\bibitem[Kne55]{kneser}
\textsc{Martin Kneser} (1955).
\newblock Aufgabe 360.
\newblock \emph{Jahresbericht der Deutschen Mathematiker-Vereinigung}
  \textbf{58}(2).

\bibitem[Lov78]{lovasz}
\textsc{L\'{a}szl\'{o} Lov\'{a}sz} (Nov. 1978).
\newblock \href{https://doi.org/10.1016/0097-3165(78)90022-5}{Kneser's
  conjecture, chromatic number, and homotopy}.
\newblock \emph{Journal of Combinatorial Theory, Series A} \textbf{25}(3),
  319--324.

\bibitem[NSS23]{NSS23}
\textsc{Tung Nguyen}, \textsc{Alex Scott}, and \textsc{Paul Seymour} (Jun.
  2023).
\newblock \href{http://arxiv.org/abs/2306.02364}{Some results and problems on
  tournament structure}.
\newblock arXiv:2306.02364.

\bibitem[Sch78]{schrijver}
\textsc{Alexander Schrijver} (1978).
\newblock Vertex-critical subgraphs of {K}neser graphs.
\newblock \emph{Nieuw Archief voor Wiskunde} \textbf{26}(3), 454--461.

\bibitem[SS16]{barbados}
\textsc{Alex Scott} and \textsc{Paul Seymour} (Mar. 2016).
\newblock
  \href{https://web.math.princeton.edu/~pds/barbados2016/openproblems.html}{{O}pen
  problems for the {B}arbados {G}raph {T}heory {W}orkshop 2016, problem 1}.
\newblock
  \urlprefix\path{https://web.math.princeton.edu/~pds/barbados2016/openproblems.html}.
\newblock Accessed on 18 May 2023.

\bibitem[Zyk49]{Z49}
\textsc{A.~Zykov} (1949).
\newblock {O}n some properties of linear complexes (in {R}ussian).
\newblock \emph{Matematicheski\v{\i} Sbornik} \textbf{24}, 163--188.
\newblock For an English translation see {\em American Mathematical Society
  Translations} {\bfseries 79}, 1952.

\end{thebibliography}
}

\end{document}